\documentclass[12pt,a4paper]{article}
\usepackage{amsmath,amssymb,amsfonts}
\usepackage{amsthm}
\usepackage{mathtools}
\usepackage{marvosym}
\usepackage{authblk}
\usepackage{enumerate}
\usepackage[english]{babel}
\usepackage{enumitem}
\usepackage{todonotes}
\usepackage{indentfirst}
\usepackage{geometry}
\usepackage{graphicx}
\usepackage{tabularx}
\usepackage{tocbibind}
\usepackage{ragged2e}
\usepackage{tikz}
\usepackage{array}
\usepackage{caption}
\usepackage{setspace}
\usepackage{hyperref}
\usepackage{lipsum}
\usepackage{listings}
\usepackage{multicol,multirow}
\usepackage{xcolor,xspace}
\usepackage[english]{babel}
\usepackage[export]{adjustbox}

\newtheorem{theorem}{Theorem}[section]
\newtheorem{corollary}{Corollary}[theorem]
\newtheorem{lemma}[theorem]{Lemma}
\newtheorem{prop}[theorem]{Proposition}

\theoremstyle{definition}
\newtheorem{definition}{Definition}[section]
\theoremstyle{remark}

\title{\textbf{On the corona product of duplication signed graph and its application}}
\author[1]{Bishal Sonar\thanks{Email: bsonarnits@gmail.com}}

\author[2]{Ravi Srivastava\thanks{Corresponding author Email: ravi@nitsikkim.ac.in}}
\affil[1,2]{Department of Mathematics, National Institute of Technology Sikkim, Barfung Block, Ravangla Sub-division, District: Namchi, Sikkim - 737139, India}
\date{}
\begin{document}
\setlength{\parindent}{0pt}
\setlength{\parskip}{\baselineskip}
\maketitle

\begin{abstract}
    \noindent This paper introduces the concepts of $\mu$-signed graphs and duplication-signed graphs and demonstrates that both types are always structurally balanced. Using the framework of duplication-signed graphs, we define two new corona products: the duplication add vertex corona product and the duplication vertex corona product, and also show that they are switching isomorphic. The structural properties of these products are explored in detail.
    Additionally, we compute the adjacency spectrum for both corona products for any pair of signed graphs, $\Gamma_1$ and $\Gamma_2$, and derive the Laplacian and signless Laplacian spectra for regular $\Gamma_1$ and arbitrary $\Gamma_2$, expressed in terms of the spectra of the constituent signed graphs.
\end{abstract}
\textbf{keywords:} duplication signed graph, $\mu$-signed graph, spectrum, integral signed graph, equienergetic signed graphs.

\section{Introduction}
    Signed graphs are an important generalization of classical graph theory, where the edges of a graph are assigned a positive or negative sign. They have been studied very widely due to their applications in modeling structural balance in social networks, circuit theory, and various optimization problems. In the study of signed graphs, their spectral properties also play an important role as tools for understanding the structure and combinatorial features of these graphs. Let $\Gamma=(G,\sigma)$ be a signed graph with a vertex set $V=\{u_1,u_2,\ldots,u_n\}$ and an edge set $E=\{e_1,e_2,\ldots,e_m\}$, where $G=(V, E)$ is the underlying unsigned graph, $\sigma: E\rightarrow \{+,-\}$ is a mapping that assigns either a positive or a negative sign to the edges of $\Gamma$, known as the signature of $\Gamma$, and $n$ and $m$ are the order and size of $\Gamma$, respectively. The signed degree of vertex $u_i$ is denoted by $sdeg(u_i)$, and it is given by $d_{u_i}^+-d_{u_i}^-$, where $d_{u_i}^+$ and $d_{u_i}^-$ are the number of positive and negative edges incident to $u_i$ respectively. The degree of vertex $u_i$ is given by $d_{u_i}=d_{u_i}^++d_{u_i}^-$. A signed graph $\Gamma$ is called \textit{net-regular} with net degree $k(k\in\mathbb{Z})$ if the signed degree of all the vertices of $\Gamma$ are $k$~\cite{nayak2016net}. In addition, $\Gamma$ is called \textit{co-regular} if the underlying unsigned graph $G$ is r-regular. The pair $(r,k)$ is called the co-regularity pair~\cite{shahul2015co}.
    A marking is a function $\mu:V\rightarrow\{1,-1\}$ which assigns either $+1$ or $-1$ to the vertices of $\Gamma$ and thus a signed graph is a $3$-tuple $\Gamma=(G,\sigma,\mu)$. Throughout this paper, we will discuss \textit{canonical marking}, although multiple marking functions of signed graphs exist~\cite {adhikari2023corona}. For canonical marking $\mu(u_i)=\prod_{e\in E_{u_i}}\sigma(e)$, where $E_{u_i}$ is the set of edges adjacent to $u_i$.
    Let $\Gamma_1=(\mathcal{G}_1,\sigma_1,\mu_1)$ and $\Gamma_2=(\mathcal{G}_2,\sigma_2,\mu_2)$ be two signed graphs. If $V(\mathcal{G}_1)=\{u_1,u_2,\ldots,u_{n_1}\}$ and $V(\mathcal{G}_2)=\{v_1,v_2,\ldots,v_{n_2}\}$ then the marking vector of $\Gamma_1$ and $\Gamma_2$ are column vectors of dimension $n_1\times1$ and $n_2\times1$ given by $\mu(\Gamma_1)=[\mu_1(u_1),\mu_1(u_2),\ldots,\mu_1(u_{n_1})]^T$ and $\mu(\Gamma_2)=[\mu_2(v_1),\mu_2(v_2),\ldots,\mu_2(v_{n_2})]^T$ respectively, where $A^T$ denotes transpose of $A$.
    Also we consider $\phi(\Gamma_j)=diag[\mu(\Gamma_j)]^T$ so that $\phi(\Gamma_j)^2=I_{n_j}$ for $i=1,2.$
    A cycle (C) is positive if the product of the signatures of its edges is positive. Otherwise, it is negative. A graph is balanced if all its cycles are positive~\cite{harary1953notion}.
    The adjacency matrix of $\Gamma$, denoted by $A(\Gamma)=(a^{\sigma}_{ij})_{n\times n}$, is an $n\times n$ matrix, where $a_{ij}^\sigma=\sigma(u_iu_j)a_{ij}$ and $a_{ij}=1$ if $u_i$ is adjacent to $u_j$, and $0$ otherwise. The Laplacian matrix $L(\Gamma)$ and the signless Laplacian matrix $Q(\Gamma)$ are given by $L(\Gamma)=D(\Gamma)-A(\Gamma)$ and $Q(\Gamma)=D(\Gamma)+A(\Gamma)$, where $D(\Gamma)=diag(d_{u_1},d_{u_2},\ldots,d_{u_n})$, a diagonal matrix with entries $d_{u_i}$. Given the same underlying graph $\mathcal{G}$, two signed graphs, $\Gamma_1=(\mathcal{G},\sigma_1)$ and $\Gamma_2=(\mathcal{G},\sigma_2)$, are switching equivalent (referred to as $\Gamma_1 \sim \Gamma_2$) if there exists a switching function $\theta$ such that for each edge $e=(u,v)$ in $\mathcal{G}$, $\theta(u)\sigma_1(e)\theta(v)=\sigma_2(e)$.
    Two signed graphs are called $M$-cospectral if they have same $M$-spectrum, $M\in\{A(\Gamma),L(\Gamma),Q(\Gamma)\}.$
    The sum of absolute eigenvalues of $A(\Gamma)$ is called the energy of $\Gamma$~\cite{bhat2015equienergetic} and is denoted by $E_{\Gamma}=\sum_{i=1}^n|\lambda_i|,$ where $\{\lambda_i:i=1,2,\ldots,n\}$ denotes the the eigenvalues of $A(\Gamma)$.\\

    In 1970, R Frucht and F Harary~\cite{frucht1970corona} first defined the \textit{corona product} for unsigned graphs; later, Cam McLeman and Erin McNicholas~\cite{mcleman2011spectra} introduced the coronal of unsigned graphs to obtain the adjacency spectra for arbitrary graphs. Then Shu and Gui~\cite{cui2012spectrum} generalized it and defined the coronal of the Laplacian and the signless Laplacian matrix of unsigned graphs. R.P. Varghese and D. Susha~\cite{varghese2017spectrum} introduced the duplication add vertex corona and duplication vertex corona spectra for unsigned graphs. We have built upon their research by exploring these spectra for signed graphs and discovered that the two products are switching isomorphic. Additionally, we have analyzed the structural properties of the product using the framework established by Adhikari et al.~\cite{adhikari2023corona}. Lastly, we have discussed potential applications of these results to generating equienergetic and integral signed graphs. This work expands the study of signed graphs by enhancing both structural and spectral dimensions. In doing so, new opportunities open for further research in more specialized graph classes and their applications.

\section{\textbf{Preliminaries}}
    The duplication graph of the unsigned graph was defined by E Sampathkumar~\cite{sampathkumar1973duplicate}; following that, we define the duplication signed graph as follows;
    \begin{definition}[Duplication signed graph] \label{D0}
        Suppose $\Gamma(G,\sigma,\mu)$ be a signed graph with vertex set $V(\Gamma)=\{u_1,u_2,\ldots,u_n\}$ and $U(\Gamma)=\{a_1,a_2,\ldots,a_n\}$ be another set corresponding to $V(\Gamma)$ and having the same marking(i.e., $\mu(a_i)=\mu(u_i);~ i=1,2,\ldots,n$). Draw $a_i$ adjacent to all the vertices in $N(u_i)$, in $\Gamma$ (sign of new edge $e=\mu(a_i)\mu(u),~ u\in N(u_i)$) and delete the edges of $\Gamma$ only. The signed graph thus obtained is called the duplication signed graph of $\Gamma$, and we denote it by $D\Gamma$.
    \end{definition}

    Following the notion of marked strong signed graph-structured defined by E.Sampathkumar et al. ~\cite{article}, we define $\mu$-signed graph as follows;
    \begin{definition}[$\mu$-signed graph] \label{D5}
        Let $\Gamma=(G,\sigma,\mu)$ be a signed graph with marking $\mu$, then the $\mu-$signed graph $\Gamma_{\mu}=(G,\sigma_{\mu},\mu)$ is a signed graph with same underlying graph G and marking $\mu$ but the following signature($\sigma_{\mu}$);
        \begin{align*}
            \sigma_{\mu}(e)=\mu(u)\mu(v), \text{ for all } e(=uv)\in E(G).
        \end{align*}
    \end{definition}

    \begin{theorem}~\cite{harary1953notion} \label{Th1}
        A signed graph $\Gamma=(G,\sigma,\mu)$ is balanced if and only if there exists a marking $\mu$ of its vertices such that for each edge $uv$ in $G$ one has $\sigma(uv)=\mu(u)\mu(v).$
    \end{theorem}

   \begin{lemma}
       $\mu-$signed graph of a signed graph $\Gamma=(G,\sigma,\mu)$ is always balanced.
   \end{lemma}

   \begin{proof}
       Proof is obvious using Definition~\ref{D5} and Theorem~\ref{Th1}.
   \end{proof}
   \begin{lemma}\label{L1}\cite{cvetkovic1980spectra}
        Let $A=\begin{bmatrix}
            A_1&A_2\\A_2&A_1
        \end{bmatrix}$
        be a block symmetric matrix of order $2\times2$. Then the eigenvalues of $A$ are those of $A_1+A_2$ together with $A_1-A_2$.
    \end{lemma}

    \begin{lemma} \label{L3}
        The duplication signed graph of any graph is balanced.
    \end{lemma}
    \begin{proof}
        Let $\Gamma=(G,\sigma,\mu)$ be a signed graph, and $D\Gamma$ be its duplication signed graph, then the Laplacian matrix of $D\Gamma$ is,
        $$L(D\Gamma)=\begin{bmatrix}
            D(\Gamma_{\mu})& -A(\Gamma_{\mu})\\
            -A(\Gamma_{\mu})& D(\Gamma_{\mu})
        \end{bmatrix}$$
        Using Lemma~\ref{L1} we have the eigenvalues of $L(D\Gamma)$ are those of $D(\Gamma_{\mu})-A(\Gamma_{\mu})=L(\Gamma_{\mu})$ along with the eigenvalues of $D(\Gamma_{\mu})+A(\Gamma_{\mu})$.\\
        So, $0$ is one eigenvalue of $L(D\Gamma)$ as $\Gamma_{\mu}$ is balanced.\\
        Hence, the duplication signed graph $D\Gamma$ is balanced.
    \end{proof}

    \begin{lemma}(Schur complement)\label{L2}~\cite{bapat2010graphs}
        Let $A_1,A_2,A_3,$ and $A_4$ be matrix of order $n_1\times n_1, n_1\times n_2, n_2\times n_1, n_2\times n_2$ respectively with $A_1$ and $A_4$ are invertible. Then 
            \begin{align*}
                \det\begin{bmatrix}
                A_1&A_2\\A_3&A_4
        \end{bmatrix}&=\det(A_1)\det(A_4-A_3A_1^{-1}A_2)\\
        &=\det(A_4)\det(A_1-A_2A_4^{-1}A_3).
        \end{align*}
    \end{lemma}
    
    \begin{definition}\cite{neumaier1992horn}[Kronecker product]
        Let $A=(a_{ij})_{n_1\times m_1}$ and $B=(b_{ij})_{n_2\times m_2}$, then the Kronecker product $A\otimes B$ of matrix $A$ and $B$ is a $n_1n_2\times m_1m_2$ matrix formed by replacing each $a_{ij}$ by $a_{ij}B$.
    \end{definition}
    The Kronecker product has several important properties:
    \begin{enumerate}
        \item It is associative: for any matrices \( A, B, C, \) and \( D \), we have \( (A \otimes B) \otimes C = A \otimes (B \otimes C) \).
        \item The transpose of the Kronecker product can be expressed as: 
        $(A \otimes B)^T = A^T \otimes B^T.$
        \item When multiplying Kronecker products, the following holds:
        $(A \otimes B)(C \otimes D) = AC \otimes BD,$ as long as the matrix products \( AC \) and \( BD \) are defined.
        \item For non-singular matrices \( A \) and \( B \), the inverse of the Kronecker product is:
        $(A \otimes B)^{-1} = A^{-1} \otimes B^{-1}.$
        \item If \( A \) is an \( n \times n \) matrix and \( B \) is an \( m \times m \) matrix, the determinant of the Kronecker product is given by:
        $\text{det}(A \otimes B) = (\text{det} A)^m (\text{det} B)^n.$
    \end{enumerate}

\section{Corona product of duplication signed graph}
    \begin{definition}[\textbf{Duplication add vertex corona product}]
        Let $\Gamma_1$ and $\Gamma_2$ be two vertex-disjoint signed graphs with $n_1$ and $n_2$ vertices, respectively. Let $D\Gamma_1$ be the \textit{duplication} signed graph of $\Gamma_1$ with vertex set $V(\Gamma_1)\cup U(\Gamma_1)$, where $V(\Gamma_1)=\{u_1,u_2,\ldots,u_{n_1}\}$ and $U(\Gamma_1)=\{a_1,a_2,\ldots,a_{n_1}\}$. Duplication add vertex corona, $\Gamma_1 \circledast \Gamma_2$, is the graph obtained from $D\Gamma_1$ and $n_1$ copies of $\Gamma_2$ by making $a_i$ adjacent to every vertices in the $i^{th}$ copy of $\Gamma_2$ for $i=1,2,\ldots,n_1.$
    \end{definition}

    \begin{definition}[\textbf{Duplication vertex corona product}]
        Let $\Gamma_1$ and $\Gamma_2$ be two vertex-disjoint signed graphs with $n_1$ and $n_2$ vertices, respectively. Let $D\Gamma_1$ be the \textit{duplication} signed graph of $\Gamma_1$ with vertex set $V(\Gamma_1)\cup U(\Gamma_1)$, where $V(\Gamma_1)=\{u_1,u_2,\ldots,u_{n_1}\}$ and $U(\Gamma_1)=\{a_1,a_2,\ldots,a_{n_1}\}$. Duplication vertex corona, $\Gamma_1 \circledcirc \Gamma_2$, is the graph obtained from $D\Gamma_1$ and $n_1$ copies of $\Gamma_2$ by making $v_i$ adjacent to every vertices in the $i^{th}$ copy of $\Gamma_2$ for $i=1,2,\ldots,n_1.$
    \end{definition}

    \begin{theorem}\label{TH0}
        The duplication add vertex corona product $\Gamma_1 \circledast \Gamma_2$ and duplication vertex corona product $\Gamma_1 \circledcirc \Gamma_2$ are switching isomorphic.
    \end{theorem}
    \begin{proof}
        Let $f:V(\Gamma_1 \circledast \Gamma_2)\rightarrow V(\Gamma_1 \circledcirc \Gamma_2)$ defined by,
        \begin{equation*}
        \begin{split}
            &f(u_i)=a_i;~~for~~i=1,2,\ldots,n_1\\
            &f(a_i)=u_i;~~for~~i=1,2,\ldots,n_1\\
            &f(v_k)=v_k;~~for~~k=1,2,\ldots,n_2.
        \end{split}
        \end{equation*}
        From Definition~\ref{D0}, for each $u_i$, there exists a corresponding $a_i$ and vice versa, so the above function is well-defined and bijective. Also, for edges $(u_i,a_j)$, $(a_i,v_k)$, and $(u_i,v_k)$ in $\Gamma_1 \circledast \Gamma_2$ there exists edges $(a_j,u_i)$, $(u_i,v_k)$, and $(a_i,v_k)$ in $\Gamma_1 \circledcirc \Gamma_2$, respectively. Hence, the underlying graphs are isomorphic.\\
        Next, for each edge $e=(uv)$ in $\Gamma_1 \circledast \Gamma_2$, and marking $\theta:V(\Gamma_1 \circledast \Gamma_2)\rightarrow \{+1\}$,
        $$\theta(u)\sigma_{\Gamma_1 \circledast \Gamma_2}(e)\theta(v)=
        \begin{cases}
            \sigma_{\Gamma_1 \circledcirc \Gamma_2}((u_i,a_j));&\text{ if }e=(u_i,a_j)\\
            \sigma_{\Gamma_1 \circledcirc \Gamma_2}((v_k,u_i));&\text{ if }e=(v_k,a_i)\\
            \sigma_{\Gamma_1 \circledcirc \Gamma_2}((v_k,v_l));&\text{ if }e=(v_k,v_l).
        \end{cases}$$
        Hence, the two signed graph products are switching isomorphic.
    \end{proof}
    Following Theorem \ref{TH0}, all the results of the subsequent section are true for both products.
    
     \begin{figure}[ht]
        \centering
        \includegraphics[scale=0.6]{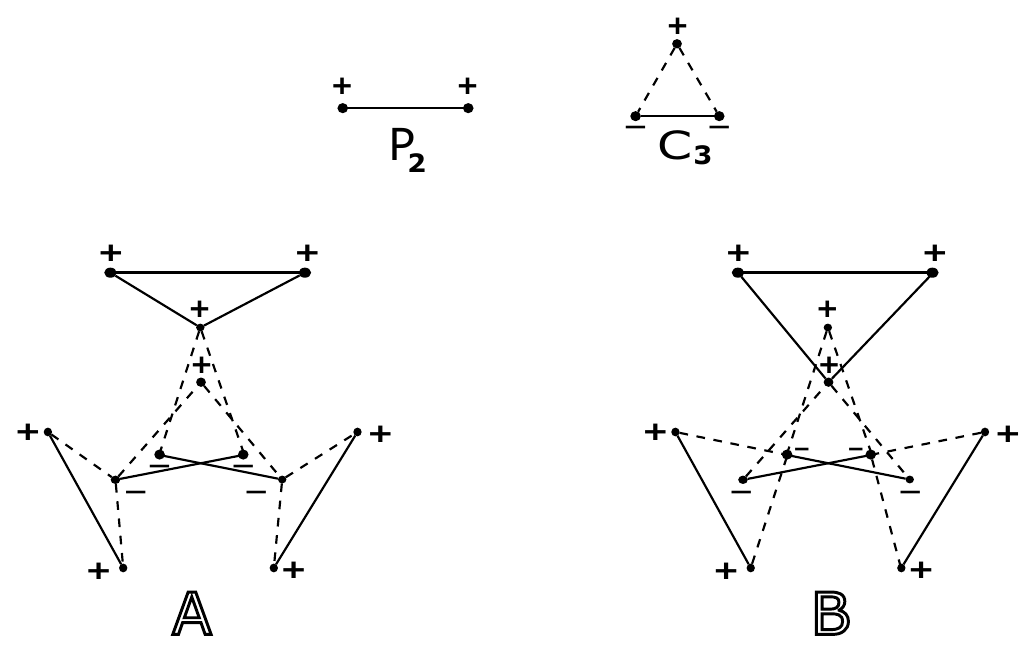}
        \caption{$A=P_2\circledast C_3$ and $B=P_2\circledcirc C_3$. The normal line implies a positive edge, and the broken line implies a negative edge.}
        \label{f1}
    \end{figure}

    \subsection{Structural properties}
        The structural property of the corona of signed graphs has already 
        being established by Adhikari at el.~\cite{adhikari2023corona}, following that we present statistical information concerning the counts of vertices, edges, and traits in $\Gamma_1\circledast\Gamma_2$ for signed graphs $\Gamma_1=(G_1,\sigma_1,\mu_1)$ and $\Gamma_2=(G_2,\sigma_2,\mu_2)$. Clearly, the number of vertices in  $\Gamma_1\circledast\Gamma_2$ is $|V|=|V_1|(2+|V_2|)$.
    \begin{align*}
		N_1^+	&= \text{count of vertices marked positive in $D\Gamma_1$}.\\
            N_2^+	&= \text{count of vertices marked positive in $\Gamma_2$}.\\
            N_1^-	&= \text{count of vertices marked negative in $D\Gamma_1$}.\\
		N_2^-	&= \text{count of vertices marked negative in $\Gamma_2$}.\\
            |DE|    &= \text{count of edges in $D\Gamma_1$}.\\
            |DE^+|  &= \text{count of positive edges in $D\Gamma_1$}.\\
            |DE^-|  &= \text{count of negative edges in $D\Gamma_1$}.\\
            |E_2^+| &= \text{count of the positive edges in } \Gamma_2.\\
            |E_2^-| &= \text{count of the negative edges in } \Gamma_2.
	\end{align*}
 
    Then Table\ref{T1}, and Table \ref{T2} describes the statistics of edges and triads in $\Gamma_1\circledast\Gamma_2$.\\
    \begin{table}[ht]
        \centering
        \begin{tabular}{|c|c|c|c|}
        \hline
        Edges & $D\Gamma_1$  & $\Gamma_2$  & $\Gamma_1\circledast\Gamma_2$ \\ \hline\hline
        Count of edges & $|DE|$ & $|E_2|$ & $|DE|+|U(\Gamma_1)||E_2|+|V_1|V_2|$ \\ \hline
        Count of positive edges & $|DE^+|$ & $|E_2^+|$ &$|DE^+|+|U(\Gamma_1)||E_2^+|+N_1^+N_2^+ +N^-_1N_2^-$\\ \hline
        Count of negative edges & $|DE^-|$ & $|E_2^-|$ &$|DE^-|+|U(\Gamma_1)||E_2^-|+N_1^+N_2^- +N^-_1N_2^+$\\ \hline
        \end{tabular}
        \caption{Count of edges in $\Gamma_1\circledast\Gamma_2.$}
        \label{T1}
    \end{table}

     \begin{table}[ht]
        \centering
        \begin{tabular}{|c|c|c|c|} \hline 
            Triads &$D\Gamma_1$ &$\Gamma_2$ &$\Gamma_1\circledast\Gamma_2$\\ \hline\hline
            Count of $T_0$ &$|T_0(D\Gamma_1)|$ &$|T_0(\Gamma_2)|$ &$|T_0(D\Gamma_1)|+|U(\Gamma_1)||T_0(\Gamma_2)|+N_u^+|E_2^+|^{\overset{+}{+}}+N_u^-|E_2^+|^{\overset{-}{-}}$  \\ \hline
            Count of $T_1$ &$|T_1(D\Gamma_1)|$ &$|T_1(\Gamma_2)|$ &$|T_1(D\Gamma_1)|+|U(\Gamma_1)||T_1(\Gamma_2)|+N_u^+(|E_2^+|^{\overset{+}{-}}+|E_2^-|^{\overset{+}{+}})$\\ &&& $+N_u^-(|E_2^+|^{\overset{+}{-}}+|E_2^-|^{\overset{-}{-}})$\\ \hline
            Count of $T_2$ &$|T_2(D\Gamma_1)|$ &$|T_2(\Gamma_2)|$ &$|T_2(D\Gamma_1)|+|U(\Gamma_1)||T_2(\Gamma_2)|+N_u^+(|E_2^+|^{\overset{-}{-}}+|E_2^-|^{\overset{+}{-}})$\\&&&$+N_u^-(|E_2^+|^{\overset{+}{+}}+|E_2^-|^{\overset{+}{-}})$\\ \hline
            Count of $T_3$ &$|T_3(D\Gamma_1)|$ &$|T_3(\Gamma_2)|$  & $|T_3(D\Gamma_1)|+|U(\Gamma_1)||T_3(\Gamma_2)|+N_u^+|E_2^-|^{\overset{-}{-}}+N_u^-|E_2^-|^{\overset{+}{+}}$\\ \hline
            \end{tabular}
        \caption{Count of Triads in $\Gamma_1\circledast\Gamma_2.$}
        \label{T2}
    \end{table}
    
    Let $p=+,-$ and denote,\\
    $|E^p_2|^{\overset{+}{+}}=$ count of edges with sign $p$ which connects a pair of positive vertices in $\Gamma_2$.\\
    $|E^p_2|^{\overset{+}{-}}=$ count of edges of sign $p$ which connects a pair opposite sign vertices in $\Gamma_2$.\\
    $|E^p_2|^{\overset{-}{-}}=$ count of edges with sign $p$ which connects a pair of negatively marked vertices in $\Gamma_2$.\\
    $T_i=$ triad with $i$ negative edges.\\
    Also let  $N_u^+\in N_1^+\cap U(\Gamma_1)$, and $N_u^-\in N_1^-\cap U(\Gamma_1)$.

    The proofs of the formulas presented in Table \ref{T1}and Table \ref{T2} directly follow from the definition of duplication add vertex corona product, and can easily be verified. Total triads of $\Gamma_1\circledast\Gamma_2$ is given by \[|T(\Gamma_1\circledast\Gamma_2)|=|T(D\Gamma_1)|+|U(\Gamma_1)|\big(|T(\Gamma_2)|+|E_2|\big)\]

    From Lemma~\ref{L3}, it is clear that $\Gamma_1\circledast\Gamma_2$ is unbalanced if $\Gamma_2$ is unbalanced. However, if $\Gamma_2$ is balanced, then it is not certain that $\Gamma_1\circledast\Gamma_2$ is balanced. We can see the same in the Figure~\ref{F2}.\\
    \begin{figure} [ht]
        \centering
        \includegraphics[scale=0.7]{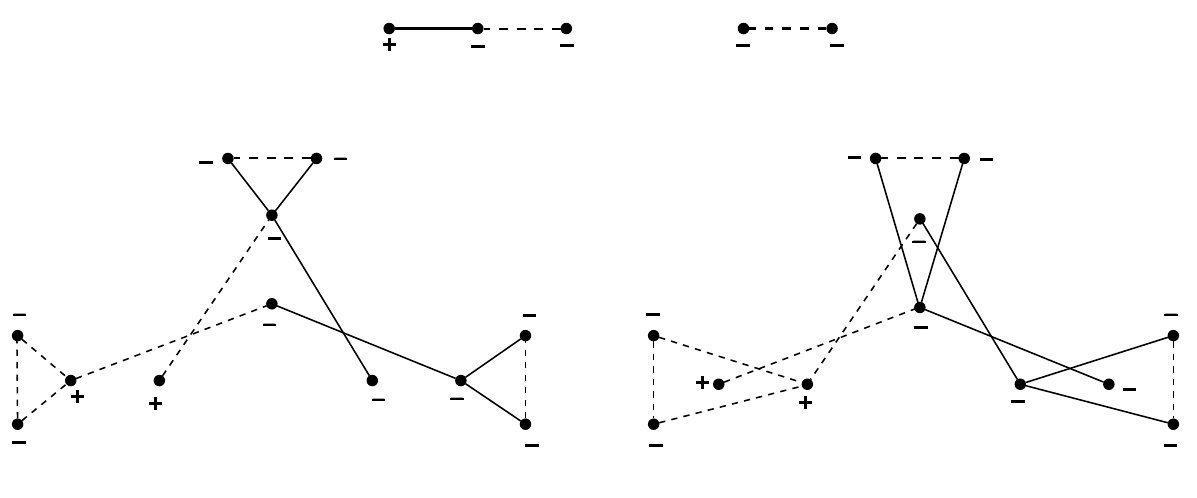}
        \caption{$P_2$ is balanced but both $P_3\circledast P_2$ and $P_3\circledcirc P_2$ are unbalanced.}
        \label{F2}
    \end{figure}

    The upcoming theorem will establish the criteria under which the signed graph $\Gamma_1\circledast\Gamma_2$ is considered unbalanced.
    \begin{theorem}
        Let $\Gamma_1=(G_1,\sigma_1,\mu_1)$ and $\Gamma_2=(G_2,\sigma_2,\mu_2)$ be two balanced signed graphs. Then $\Gamma_1\circledast\Gamma_2$ is unbalanced if and only if $\Gamma_2$ contains the following types of edges.
        \begin{enumerate}
            \item positive edge connecting opposite marked vertices.
            \item negative edge connecting a pair of positive vertices.
            \item negative edge connecting a pair of negative vertices.
        \end{enumerate}
    \end{theorem}
    \begin{proof}
        The proof is the consequence of the observation that any positive or negative marked vertices of $\Gamma_1$ form a triad $T_1$ in $\Gamma_1\circledast\Gamma_2$ when there are edges of type $(i)$ and/or $(ii)$ in $\Gamma_2$. Otherwise, it forms a triad $T_3$ when $\Gamma_2$ includes an edge of type $(iii).$
    \end{proof}
    
    \subsection{Spectral properties}
        Cam McLeman and Erin McNicholas~\cite{mcleman2011spectra} first introduced the coronal of graphs. Then Shu and Gui~\cite{cui2012spectrum} generalized it and defined coronal for Laplacian and signless Laplacian matrix of unsigned graphs. Recently, Amrik at. el.~\cite{Structuralbalanceandsignedgraphs_2025} defined the coronal of the signed graph as follows;
    \begin{definition}\cite{Structuralbalanceandsignedgraphs_2025}
        Given a signed graph $\Gamma=(G,\sigma,\mu)$ of order $n$ and $M$ be a graph matrix of $\Gamma$. Now, viewed as a matrix over the rational function field $\mathbb{C}(\lambda)$, the characteristic matrix $\lambda I_n-M$ has a non-zero determinant, which is invertible. The signed $M$-coronal $\chi_M(\lambda) \in \mathbb{C}(\lambda)$ of $\Gamma$ is defined as, 
        \begin{equation}\label{eqn1}
        \chi_M(\lambda)=\mu(\Gamma)^T(\lambda I_n-M)^{-1}\mu(\Gamma)
        \end{equation}
    \end{definition}
    If we replace $M$ in (\ref{eqn1}) by $A(\Gamma)$, $L(\Gamma)$, or $Q(\Gamma)$ then we will obtain signed $A(\Gamma)$-coronal, signed $L(\Gamma)$-coronal, or signed $Q(\Gamma)$-coronal of $\Gamma$ respectively.\\
    
    Using the above notions for two signed graphs $\Gamma_1$ and $\Gamma_2$ with $n_1$ and $n_2$ vertices, we have the following;\vspace{1pt}
   \begin{equation}\label{Eqn1}
           \big(\mu(\Gamma_2^T)\otimes\phi(\Gamma_1)\big)\big((xI_{n_2}-M(\Gamma_2))^{-1}\otimes I_{n_1})\big)\big(\mu(\Gamma_2)\otimes \phi(\Gamma_1)\big)=\chi_{M(\Gamma_2)}(x)I_{n_1} \vspace{2pt}
    \end{equation}
    \begin{equation}\label{Eqn2}
        \big(\mu(\Gamma_2^T)\otimes\phi(\Gamma_1)\big)\big(((x-1)I_{n_2}-M(\Gamma_2))^{-1}\otimes I_{n_1})\big)\big(\mu(\Gamma_2)\otimes \phi(\Gamma_1)\big)=\chi_{M(\Gamma_2)}(x-1)I_{n_1}.
    \end{equation}

    Clearly, for a signed graph $\Gamma$ with order $n$, \begin{equation}
        \begin{split}
            \chi_{M(\Gamma)}(\lambda)&=\frac{\mu(\Gamma)^TAdj(\lambda I_n-M(\lambda))\mu(\Gamma)}{\det(\lambda I_n-M(\Gamma))}\\
            &=\frac{p(M(\Gamma),\lambda)}{f(M(\Gamma),\lambda)},
        \end{split}
    \end{equation}

    where $p(A(\Gamma),\lambda)$ is a polynomial of degree $n-1$ and $f(M(\Gamma),\lambda)$ is the characteristics polynomial of the matrix $M$ of the signed graph $\Gamma$. If the greatest common divisor of these polynomials is a constant, then the above polynomial ratio can be further simplified as follows: \begin{equation}\label{Equ5}
        =\frac{P_{d-1}(\lambda)}{F_d(\lambda)},
    \end{equation}
    where $P_{d-1}(\lambda)$ and $F_d(\lambda)$ are polynomials of degree $d-1$ and $d$ respectively and $$gcd\big(p(M(\Gamma),\lambda),f(M(\Gamma),\lambda)\big)=R_{n-d}(\lambda)$$
    of degree $n-d$.
    
    Let $\Gamma_1=(G_1,\sigma_1,\mu_1)$ and $\Gamma_2=(G_2,\sigma_2,\mu_2)$ be arbitrary signed graphs on $n_1$ and $n_2$ vertices, respectively. Following~\cite{gopalapillai2011spectrum}, we first label the vertices of $\Gamma_1\circledast\Gamma_2$ as follows. Let $V(\Gamma_1)=\{u_1,u_2,\ldots,u_{n_1}\},$ $U(\Gamma_1)=\{a_1,a_2,\ldots,a_{n_1}\},$ and $V(\Gamma_2)=\{v_1,v_2,\ldots,v_{n_2}\}$. For $i=1,2,\ldots,n_1,$ let $v_1 ^i,v_2 ^i,\ldots,v_{n_2} ^i$ denotes the vertices of the $i^{th}$ copy of $\Gamma_2$, with the understanding that $v_j^i$ is the copy of $a_j$. Denote
    \begin{equation} \label{eqn2}
                W_j=\big\{v_j^1,v_j^2,\ldots,v_j^{n_1}\big\}, ~~~j=1,2,\ldots,n_2
    \end{equation}
    Then $V(D\Gamma_1)\cup W_1\cup W_2\cup\ldots\cup W_{n_2}$ is a partition of $V(\Gamma_1\circledast\Gamma_2)$ .\\
    It is clear that the degrees of the vertices of $\Gamma_1\circledast\Gamma_2$ are:\\    
        $d_{\Gamma_1\circledast\Gamma_2}(a_i)=n_2+d_{\Gamma_1}(a_i), ~~i=1,2,\ldots,n_1 \\$
        $d_{\Gamma_1\circledast\Gamma_2}(u_i)=d_{\Gamma_1}(u_i), ~~i=1,2,\ldots,n_1\\$
        $d_{\Gamma_1\circledast\Gamma_2}(v^i_j)=d_{\Gamma_2}(v_j)+1, ~~i=1,2,\ldots,n_1, ~~j=1,2,\ldots,n_2.$\\

    \subsubsection{A-spectra of signed graph}
        \begin{theorem}~\label{Theorem1}
            Let $\Gamma_i$ be two signed graphs with vertices $n_i$ and eigenvalues $\lambda_{i1}\geq\lambda_{i2}\geq\ldots\geq\lambda_{in_i}$, for $i=1,2$. Let $\Gamma_{1\mu}$ be the $\mu$-signed graph of $\Gamma_1$, and $\lambda'_{11},\lambda'_{12},\ldots,\lambda'_{1n_1}$ be the eigenvalue of $\Gamma_{1\mu}.$ Then the adjacency characteristics polynomial of duplication add vertex corona $\Gamma_1\circledast\Gamma_2$ is
            $$f(A(\Gamma_1\circledast\Gamma_2),x)=\prod_{i=1}^{n_2}(x-\lambda_{2i})^{n_1} \prod_{i=1}^{n_1}(x^2-x\chi_{A(\Gamma_2)}(x)-{\lambda'}_{1i}^2).$$
        \end{theorem}
        \begin{proof}
            With respect to the partition mentioned in (\ref{eqn2}), the adjacency matrix of $\Gamma_1\circledast\Gamma_2$ can be written as, 
        \[A(\Gamma_1\circledast\Gamma_2)=\begin{bmatrix}
            0_{n_1}& A(\Gamma_{1\mu})& 0_{n_1\times n_1n_2}\\ 
            A(\Gamma_{1\mu})& 0_{n_1}& \mu(\Gamma_2)^T\otimes\phi(\Gamma_1)\\
            0_{n_1n_2\times n_1}& \mu(\Gamma_2)\otimes\phi(\Gamma_1)& A(\Gamma_2)\otimes I_{n_1}
        \end{bmatrix}\]
        Then using Lemma \ref{L2} the adjacency characteristics polynomial of $\Gamma_1\circledast\Gamma_2$ is given by,
        \begin{equation*}
        \begin{split}
            f(A(\Gamma_1\circledast\Gamma_2);x)&= \det\big(xI_{n_1(2+n_2)}-A(\Gamma_1\circledast\Gamma_2)\big)\\
            &=\begin{vmatrix}
                xI_{n_1}& -A(\Gamma_{1\mu})& 0_{n_1\times n_1n_2}\\
                -A(\Gamma_{1\mu})& xI_{n_1}& -\mu(\Gamma_2)^T\otimes\phi(\Gamma_1)\\
                0_{n_1n_2\times n_1}& -\mu(\Gamma_2)\otimes\phi(\Gamma_1)& (xI_{n_2}-A(\Gamma_2))\otimes I_{n_1})
            \end{vmatrix}\\
            &=\det\{(xI_{n_2}-A(\Gamma_2))\otimes I_{n_1}\}\det(S).
        \end{split}
        \end{equation*}
        where,
        \begin{equation*}
            \begin{split}
                S&=\begin{bmatrix}
                xI_{n_1}& -A(\Gamma_{1\mu})\\
                -A(\Gamma_{1\mu})& xI_{n_1}
            \end{bmatrix} \begin{bmatrix}
                0\\
                -\mu(\Gamma_2)^T\otimes\phi(\Gamma_1)
            \end{bmatrix} - \{(xI_{n_2}-A(\Gamma_2))^{-1}\otimes I_{n_1}\}\\
            &~~~~\begin{bmatrix}
                0& \mu(\Gamma_2)\otimes\phi(\Gamma_1)
            \end{bmatrix}
            \end{split}
        \end{equation*}
        Using Equation~(\ref{Eqn1}) in the above equation we get: -
        \begin{equation*}
            \begin{split}
            S&=\begin{bmatrix}
                xI_{n_1}& -A(\Gamma_{1\mu})\\
                -A(\Gamma_{1\mu})& xI_{n_1}
            \end{bmatrix}-\begin{bmatrix}
                0& 0\\
                0& \chi_{A(\Gamma_2)}(x)I_{n_1}
            \end{bmatrix}\\
            &=\begin{bmatrix}
                 xI_{n_1}& -A(\Gamma_{1\mu})\\
                -A(\Gamma_{1\mu})& (x-\chi_{A(\Gamma_2)}(x))I_{n_1}
            \end{bmatrix}
            \end{split}
        \end{equation*}
        Therefore,
        \begin{equation*}
            \begin{split}
                 \det S&=\det(xI_{n_1})\det\{(x-\chi_{A(\Gamma_2)}(x))I_{n_1}-A(\Gamma_{1\mu})(xI_{n_1})^{-1}A(\Gamma_{1\mu})\}\\
                &=x^{n_1}\det \{(x-\chi_{A(\Gamma_2)}(x))I_{n_1}-\frac{A(\Gamma_{1\mu})^2}{x}\}\\
                &=\prod_{i=1}^{n_1}\{x^2-x\chi_{A(\Gamma_2)}(x)-{\lambda'}_{1i}^2\}.
            \end{split}
        \end{equation*}
        Also, by the property of the Kronecker product,
        \begin{equation*}
            \begin{split}
                \det (xI_{n_2}-A(\Gamma_2)\otimes I_{n_1})&=\big(\det (xI_{n_2}-A(\Gamma_2))\big)^{n_1} (\det(I_{n_1}))^{n_2}\\
                &=\prod_{i=1}^{n_2}(x-\lambda_{2i})^{n_1}.
            \end{split}
        \end{equation*}
        Hence, $f(A(\Gamma_1\circledast\Gamma_2),x)=\prod_{i=1}^{n_2}(x-\lambda_{2i})^{n_1} \prod_{i=1}^{n_1}(x^2-x\chi_{A(\Gamma_2)}(x)-{\lambda'}_{1i}^2).$
        \end{proof}

        \begin{corollary} \label{C0}
            If $\Gamma_1$ is balanced, then $\Gamma_1$ and $~\Gamma_{1\mu}$ are co-spectral and the characteristics polynomial of $\Gamma_1\circledast\Gamma_2$ is $$f(A(\Gamma_1\circledast\Gamma_2),x)=\prod_{i=1}^{n_2}(x-\lambda_{2i})^{n_1} \prod_{i=1}^{n_1}(x^2-x\chi_{A(\Gamma_2)}(x)-\lambda_{1i}^{2}).$$
        \end{corollary}
    We use the following result by Amrik et.al.~\cite{adhikari2023corona} in the following corollary.\\
     Let $\Gamma=(G,\sigma,\mu)$ be a non-null, $(r,k)$-co-regular signed graph on $n$ nodes, then signed $A(\Gamma)-$coronal of $\Gamma$ is, \[\chi_{A(\Gamma)}(\lambda)=\frac{n}{\lambda-k}\].

     \begin{corollary} \label{C1}
         Let $\Gamma_1$ be a signed graph with $n_1$ vertices, and $\Gamma_2$ be a $(r,k)-$co-regular signed graph with $n_2$ vertices. Also let the multiplicity of eigenvalue $k$ of $\Gamma_2$  be $p$, then the A-spectrum of $\Gamma_1\circledast\Gamma_2$ consist of,
         \begin{enumerate}
             \item $\lambda_{2i}$, repeated $n_1$ times, for each $i=2,3,\ldots,n_2$.
             \item Three roots of the equation $x^3-kx^2-(n_2+{\lambda'}_{1i}^2)x+k{\lambda'}_{1i}^2=0$, for each eigenvalue ${\lambda'}_{1i}(i=1,2,\ldots,n_1)$ of $\Gamma_{1\mu}$.
             \item Eigenvalue $k$ with multiplicity $n_1(p-1)$.
         \end{enumerate}
    \end{corollary}

    \begin{proof}
        From above result, $$\chi_{A(\Gamma_2)}(x)=\frac{n_2}{x-k}$$
        The only pole of $\chi_{A(\Gamma_2)}(x)$ is $x=k$, which is equivalent to the eigenvalue $x=k=\lambda_{21}$(say) of $\Gamma_2$. So by Theorem~\ref{Theorem1} the adjacency spectrum of $\Gamma_1\circledast\Gamma_2$ is given by,
        \begin{enumerate}
            \item $\lambda_{2i}$, repeated $n_1$ times, for each $i=2,3,\ldots,n_2$.
            \item $3n_1$ eigenvalues are obtained by solving,
            \begin{equation*}
                \begin{split}
                    &x^2-x\chi_{A(\Gamma_2)}(x)-{\lambda'}_{1i}^2=0\\
                    \implies&x^2-\frac{n_2x}{x-k}-{\lambda'}_{1i}^2=0\\
                    \implies&x^3-kx^2-(n_2+{\lambda'}_{1i}^2)x+k{\lambda'}_{1i}^2=0
                \end{split}
            \end{equation*}
            for each eigenvalue ${\lambda'}_{1i}(i=1,2,\ldots,n_1)$ of $\Gamma_{1\mu}$.
            \item The remaining $n_1(p-1)$ are equal to the pole $x=k$ of $\chi_{A(\Gamma_2)}(x)$.
        \end{enumerate}
    \end{proof}

    Similarly, we use the following result by Bishal et al.~\cite{sonar2023spectrum} in the following corollary.\\
    Let $\Gamma=(K_{1,n},\sigma,\mu)$ be a signed star with $V(\Gamma)=\{v_1,v_2,\ldots,v_{n+1}\}$ such that $d(v_1)=n$, then \[\chi_{A(\Gamma)}(\lambda)=\frac{(n+1)\lambda+2n\mu(v_1)}{\lambda^2-n}.\]
    \begin{corollary}~\label{C2}
        Let $\Gamma_1$ be a signed graph with $n_1$ vertices and $\Gamma_2=(K_{1,n_2},\sigma_2,\mu_2)$ be a signed star on $n_2+1$ vertices with $V(\Gamma_2)=\{v_1,v_2,\ldots,v_{n_2+1}\}$ where $d(v_1)=n_2$. Then the spectrum of $\Gamma_1\circledast\Gamma_2$ consist of,
        \begin{enumerate}
            \item The eigenvalue $0$ with multiplicity $n_1(n_2-1).$
            \item The other $4n_1$ eigenvalues are obtained by solving the equation $x^4-(2n_2+1+{\lambda}_{1i}^2)x^2-2n_2\mu(v_1)x+n_2{\lambda}_{1i}^2=0,$ for each eigenvalue ${\lambda}_{1i}(i=1,2,\ldots,n_1)$ of $\Gamma_{1}$.
        \end{enumerate}
    \end{corollary}

    \begin{proof}
        The adjacency spectrum of $K_{1,n_2}$ is $(-\sqrt{n_2},\sqrt{n_2},0^{(n_2-1)}).$\\
        From the above result, $$\chi_{A(\Gamma_2)}(x)=\frac{(n_2+1)x+2n_2\mu(v_1)}{x^2-n_2}$$
        The poles of $\chi_{A(\Gamma_2)}(x)$ are $x=\pm\sqrt{n_2}$ which is equivalent to the maximal and minimal eigenvalues of $\Gamma_2$.\\
        By Corollary~\ref{C0}, the spectrum of $\Gamma_1\circledast\Gamma_2$ is given by,
        \begin{enumerate}
            \item The eigenvalue $0$ with multiplicity $n_1(n_2-1)$.
            \item The other $4n_1$ eigenvalues are obtained by solving the equation,
            \begin{equation*}
                \begin{split}
                    &x^2-x\chi_{A(K_{1,n_2})}(x)-{\lambda}_{1i}^2=0\\
                    \implies&x^2-\frac{(n_2+1)x+2n_2\mu(v_1)}{x^2-n_2}x-{\lambda}_{1i}^2=0\\
                    \implies&x^4-(2n_2+1+{\lambda}_{1i}^2)x^2-2n_2\mu(v_1)x+n_2{\lambda}_{1i}^2=0
                \end{split}
            \end{equation*}
            for each eigenvalue ${\lambda}_{1i}(i=1,2,\ldots,n_1)$ of $\Gamma_{1}$.
        \end{enumerate}
        
    \end{proof}
    
     \begin{corollary}
         Let $\Gamma_1$ and $\Gamma_2$ be two $A(\Gamma_{\mu})-$co-spectral signed graphs and $\Gamma$ is any arbitrary signed graph, then 
            \begin{enumerate}
                \item $\Gamma_1\circledast\Gamma$ and $\Gamma_2\circledast\Gamma$ are $A$-co-spectral.
                \item $\Gamma\circledast\Gamma_1$ and $\Gamma\circledast\Gamma_2$ are co-spectral if $\chi_{A(\Gamma_1)}(x)=\chi_{A(\Gamma_2)}(x).$
            \end{enumerate}
     \end{corollary}

    \subsection{L-spectra of signed graph}
    \begin{theorem}~\label{TH2}
         Let $\Gamma_1$ be a $r_1-$regular signed graph with $n_1$ vertices and $\Gamma_2$ be any arbitrary signed graph on $n_2$ vertices with Laplacian spectrum $\gamma_{i1}\leq\gamma_{i2}\leq\ldots\leq\gamma_{in_i},~~i=1,2.$ Let $\Gamma_{1\mu}$ be the $\mu-$signed graph of $\Gamma_1$, and $\lambda'_{11},\lambda'_{12},\ldots,\lambda'_{1n_1}$ be the eigenvalue of $\Gamma_{1\mu}.$ Then the Laplacian characteristics polynomial of duplication add vertex corona $\Gamma_1\circledast\Gamma_2$ is $$f(L(\Gamma_1\circledast\Gamma_2);x)=\prod_{i=1}^{n_2}(x-1-\gamma_{2i})^{n_1} \prod_{i=1}^{n_1}\big((x-r_1-n_2-\chi_{L(\Gamma_2)}(x-1))(x-r_1)-{\lambda'}_{1i}^2\big)$$
    \end{theorem}
        
    \begin{proof}
        With respect to the partition mentioned in (\ref{eqn2}), the diagonal degree matrix of $\Gamma_1\circledast\Gamma_2$ can be written as, 
        $$D(\Gamma_1\circledast\Gamma_2)=\begin{bmatrix}
            r_1I_{n_1}& 0& 0\\
            0& (r_1+n_2)I_{n_1}& 0\\
            0& 0& (D(\Gamma_2)+I_{n_2})\otimes I_{n_1}
        \end{bmatrix}$$
        Then the Laplacian matrix of $\Gamma_1\circledast\Gamma_2$ is, 
        \begin{equation*}
            \begin{split}
                L&=D-A\\
                &=\begin{bmatrix}
                  r_1I_{n_1}& -A(\Gamma_{1\mu})& 0_{n_1\times n_1n_2}\\ 
            -A(\Gamma_{1\mu})& (r_1+n_2)I_{n_1}& -\mu(\Gamma_2)^T\otimes\phi(\Gamma_1)\\
            0_{n_1n_2\times n_1}& -\mu(\Gamma_2)\otimes\phi(\Gamma_1)& (L(\Gamma_2)+I_{n_2})\otimes I_{n_1}
                \end{bmatrix}
            \end{split}
        \end{equation*}
        Therefore, the Laplacian characteristics polynomial of $\Gamma_1\circledast\Gamma_2$ is,
        \begin{equation*}
            \begin{split}
                f(L(\Gamma_1\circledast\Gamma_2);x)&=\det\begin{bmatrix}
                    (x-r_1)I_{n_1}& A(\Gamma_{1\mu})& 0_{n_1\times n_1n_2}\\ 
            A(\Gamma_{1\mu})& (x-r_1-n_2)I_{n_1}& \mu(\Gamma_2)^T\otimes\phi(\Gamma_1)\\
            0_{n_1n_2\times n_1}& \mu(\Gamma_2)\otimes\phi(\Gamma_1)& ((x-1)I_{n_2}-L(\Gamma_2))\otimes I_{n_1}
                \end{bmatrix}\\
                &=\det\big(((x-1)I_{n_2}-L(\Gamma_2))\otimes I_{n_1}\big) \det(S),\\
                \text{where } S&=\begin{bmatrix}
                    (x-r_1)I_{n_1}& A(\Gamma_{1\mu})\\ 
            A(\Gamma_{1\mu})& (x-r_1-n_2)I_{n_1}
                \end{bmatrix}-\begin{bmatrix}
                    0\\
                    \mu(\Gamma_2)^T\otimes\phi(\Gamma_1)
                \end{bmatrix} \\ &~~~~\big(((x-1)I_{n_2}-L(\Gamma_2))\otimes I_{n_1}\big)^{-1} \begin{bmatrix}
                    0_{n_1n_2\times n_1}& \mu(\Gamma_2)\otimes\phi(\Gamma_1)
                \end{bmatrix}\\
            \end{split}
        \end{equation*}
        Using Equation~(\ref{Eqn2}) in the above equation we get: -\\
        \begin{equation*}
            \begin{split}
                S&=\begin{bmatrix}
                    (x-r_1)I_{n_1}& A(\Gamma_{1\mu})\\
                    A(\Gamma_{1\mu})& (x-r_1-n_2)I_{n_1}
                \end{bmatrix}-\begin{bmatrix}
                    0& 0\\
                    0& \chi_{L(\Gamma_2)}(x-1)I_{n_1}
                \end{bmatrix}\\
                &=\begin{bmatrix}
                    (x-r_1)I_{n_1}& A(\Gamma_{1\mu})\\
                    A(\Gamma_{1\mu})& (x-r_1-n_2)I_{n_1}-\chi_{L(\Gamma_2)}(x-1)I_{n_1}
                \end{bmatrix}
            \end{split}
        \end{equation*}
        Therefore,
        \begin{equation*}
            \begin{split}
                \det S&=(x-r_1)^{n_1}\det\Big((x-r_1-n_2-\chi_{L(\Gamma_2)}(x-1))I_{n_1}-\frac{A(\Gamma_{1\mu})^2}{x-r_1}\Big)\\
                &=\prod_{i=1}^{n_1}\big((x-r_1-n_2-\chi_{L(\Gamma_2)}(x-1))(x-r_1)-{\lambda'}_{1i}^2\big)
            \end{split}
        \end{equation*}
        Also, by the property of the Kronecker product,
        \begin{equation*}
            \begin{split}
                \det\big(((x-1)I_{n_2}-L(\Gamma_2))\otimes I_{n_1}\big)&=\det(I_{n_1})^{n_2}\det((x-1)I_{n_2}-L(\Gamma_2))^{n_1}\\
                &=\prod_{i=1}^{n_2}(x-1-\gamma_{2i})^{n_1}.
            \end{split}
        \end{equation*}
        Hence, $f(L(\Gamma_1\circledast\Gamma_2);x)=\prod_{i=1}^{n_2}(x-1-\gamma_{2i})^{n_1} \prod_{i=1}^{n_1}\big((x-r_1-n_2-\chi_{L(\Gamma_2)}(x-1))(x-r_1)-{\lambda'}_{1i}^2\big).$
    \end{proof}

    \begin{corollary}.\\ \vspace{-20pt}
         \begin{enumerate}
            \item Let $\Gamma_1$ and $\Gamma_2$ be two $A(\Gamma_{\mu})$-co-spectral regular signed graphs and $\Gamma$ is any arbitrary signed graph then $\Gamma_1\circledast\Gamma$ and $\Gamma_2\circledast\Gamma$ are Laplacian co-spectral.
            \item Let $\Gamma$ be a regular signed graph and $\Gamma_1$ and $\Gamma_2$ be two Laplacian co-spectral signed graphs then $\Gamma\circledast\Gamma_1$ and $\Gamma\circledast\Gamma_2$ are Laplacian co-spectral if $\chi_{L(\Gamma_1)}(x)=\chi_{L(\Gamma_2)}(x).$
        \end{enumerate}
    \end{corollary}

    \begin{theorem}
         Let $\Gamma_1$ be a $r_1-$regular signed graph with $n_1$ vertices and $\Gamma_2$ be any arbitrary signed graph on $n_2$ vertices with signless Laplacian spectrum $\nu_{i1}\leq\nu_{i2}\leq\ldots\leq\nu_{in_i},~~i=1,2.$ Let $\Gamma_{1\mu}$ be the $\mu-$signed graph of $\Gamma_1$, and $\lambda'_{11},\lambda'_{12},\ldots,\lambda'_{1n_1}$ be the eigenvalue of $\Gamma_{1\mu}.$ Then the signless Laplacian characteristics polynomial of duplication add vertex corona $\Gamma_1\circledast\Gamma_2$ is $$f(Q(\Gamma_1\circledast\Gamma_2);x)=\prod_{i=1}^{n_2}(x-1-\nu_{2i})^{n_1} \prod_{i=1}^{n_1}\big((x-r_1-n_2-\chi_{Q(\Gamma_2)}(x-1))(x-r_1)-{\lambda'}_{1i}^2\big)$$
    \end{theorem}

    \begin{proof}
        The signless Laplacian matrix of $\Gamma_1\circledast\Gamma_2$ is,\\
        \begin{equation*}
            \begin{split}
                Q&=D+A\\
                &=\begin{bmatrix}
                  r_1I_{n_1}& A(\Gamma_{1\mu})& 0_{n_1\times n_1n_2}\\ 
            A(\Gamma_{1\mu})& (r_1+n_2)I_{n_1}& \mu(\Gamma_2)^T\otimes\phi(\Gamma_1)\\
            0_{n_1n_2\times n_1}& \mu(\Gamma_2)\otimes\phi(\Gamma_1)& (Q(\Gamma_2)+I_{n_2})\otimes I_{n_1}
                \end{bmatrix}
            \end{split}
        \end{equation*}
        The latter part of the proof is similar to Theorem~\ref{TH2}.
    \end{proof}

    \begin{corollary}.\\ \vspace{-20pt}
         \begin{enumerate}
            \item Let $\Gamma_1$ and $\Gamma_2$ be two $A(\Gamma_{\mu})-$co-spectral regular signed graphs and $\Gamma$ is any arbitrary signed graph then $\Gamma_1\circledast\Gamma$ and $\Gamma_2\circledast\Gamma$ are signless Laplacian co-spectral.
            \item Let $\Gamma$ be a regular signed graph and $\Gamma_1$ and $\Gamma_2$ be two signless Laplacian co-spectral signed graphs then $\Gamma\circledast\Gamma_1$ and $\Gamma\circledast\Gamma_2$ are signless Laplacian co-spectral if $\chi_{Q(\Gamma_1)}(x)=\chi_{Q(\Gamma_2)}(x).$
        \end{enumerate}
    \end{corollary}

    \section{Application}
    This section discusses some applications of the signed graph product $\Gamma_1\circledast\Gamma_2$ and its spectrum in generating a family of integral signed graphs and a sequence of non-co-spectral equienergetic signed graphs.
    \subsection{Families of Integral Signed Graphs}
    A signed graph is called an integral signed graph if all the eigenvalues are integers. \cite {balinska2002survey,indulal2012spectrum}. The following Propositions show the essential conditions for $\Gamma_1\circledast\Gamma_2$ to be an integral signed graph.

    \begin{prop}
        Let $\Gamma_1$ be a signed graph with $n_1$ vertices, and $\Gamma_2$ be a $k$-net-regular signed graph with $n_2$ vertices. Then $\Gamma_1\circledast\Gamma_2$ are integral signed graphs if and only if $\Gamma_{1\mu}, \Gamma_2$ are integral signed graph and the roots of the equation, $x^3-kx^2-(n_2+{\lambda'}_{1i}^2)x+k{\lambda'}_{1i}^2=0$, for each eigenvalue ${\lambda'}_{1i}(i=1,2,\ldots,n_1)$ of $\Gamma_{1\mu}$ are integers.
    \end{prop}

    \begin{proof}
        The proof follows from Corollary~\ref{C1}.
    \end{proof}
    \begin{prop}
        Let $\Gamma_1$ be a signed graph with $n_1$ vertices and $\Gamma_2=(K_{1,n_2},\sigma_2,\mu_2)$ be a signed star on $n_2+1$ vertices with $V(\Gamma_2)=\{v_1,v_2,\ldots,v_{n_2+1}\}$ where $d(v_1)=n_2$. Then $\Gamma_1\circledast\Gamma_2$ are integral signed graphs if and only if $\Gamma_{1}, ~\Gamma_2$ are integral signed graph and the roots of the equation, $x^4-(2n_2+1+{\lambda}_{1i}^2)x^2-2n_2\mu(v_1)x+n_2{\lambda}_{1i}^2=0,$ for each eigenvalue ${\lambda}_{1i}(i=1,2,\ldots,n_1)$ of $\Gamma_{1}$ are integers.
    \end{prop}

    \begin{proof}
        The proof follows from Corollary~\ref{C2}.
    \end{proof}

    \subsection{Equienergetic signed graphs}
    In the following Theorem~\ref{TH3}, we describe a method for constructing non-co-spectral equienergetic signed graphs.

    \begin{theorem}~\label{TH3}
        Let $\Gamma_1$ and $\Gamma_2$ be two non-co-spectral equienergetic signed graphs of order $m$ with the same coronal, then for any arbitrary signed graph $\Gamma$ of order $n$, the signed graphs $\Gamma\circledast\Gamma_1$ and $\Gamma\circledast\Gamma_2$ are also non-co-spectral equienergetic.
    \end{theorem}
    \begin{proof}
        Since $\Gamma_1$ and $\Gamma_2$ are having same coronal, the polynomials $P_{d-1}(\lambda)$ and $F_d(\lambda)$ for both $\Gamma_1$ and $\Gamma_2$ given by equation~(\ref{Equ5}) foe Adjacency matrix are equal.\\
        Let \begin{equation}
                f(A(\Gamma_1),\lambda)=R_{m-d}(\lambda)F_d(\lambda)
        \end{equation}
        and\begin{equation}
                f(A(\Gamma_2),\lambda)=R'_{m-d}(\lambda)F_d(\lambda).
        \end{equation}
        Clearly, $R_{m-d}$ and $R'_{m-d}$ are different, as the graphs $\Gamma_1$ and $\Gamma_2$ are non-co-spectral. The characteristics polynomial of $A(\Gamma\circledast\Gamma_1)$ is,
        \begin{equation} \label{Eqn9}
            \begin{split}
                f(A(\Gamma\circledast\Gamma_1),\lambda) &=f(A(\Gamma_1),\lambda)^{n}\prod_{i=1}^{n}\big(\lambda^2-x\chi_{A(\Gamma_1)}(\lambda)-{\lambda'}_{1i}^2\big)\\
                &=(R_{m-d}(\lambda)F_d(\lambda))^n\prod_{i=1}^n\big(\lambda^2-\lambda\frac{P_{d-1}(\lambda)}{F_d(\lambda)}-{\lambda'}^2_{1i}\big)\\
                &=R_{m-d}(\lambda)^n\prod_{i-1}^n\big(\lambda^2F_d(\lambda)-\lambda P_{d-1}(\lambda)-{\lambda'}^2_{1i}F_d(\lambda)\big).
            \end{split}
        \end{equation}
        Similarly the characteristics polynomial of $A(\Gamma\circledast\Gamma_2)$ is,
        \begin{equation} \label{Eqn10}
            f(A(\Gamma\circledast\Gamma_2),\lambda)=R'_{m-d}(\lambda)^n\prod_{i-1}^n\big(\lambda^2F_d(\lambda)-\lambda P_{d-1}(\lambda)-{\lambda'}^2_{1i}F_d(\lambda)\big).
        \end{equation}
        Let the roots of the polynomial $F_d(\lambda)$ be $\delta_1,\delta_2,\ldots,\delta_d$ and the roots of $R_{m-d}(\lambda)$ and $R'_{m-d}(\lambda)$ be $\alpha_1,\alpha_2,\ldots,\alpha_{m-d}$ and $\beta_1,\beta_2,\ldots,\beta_{m-d}$ respectively.\\
        Now since the signed graphs $\Gamma_1$ and $\Gamma_2$ are equienergetic, we have,
        \begin{equation*}
            \sum_{i=1}^d|\delta_i|+\sum_{i=1}^{m-d}|\alpha_i|=E_{\Gamma_1}=E_{\Gamma_2}=\sum_{i=1}^d|\delta_i|+\sum_{i=1}^{m-d}|\beta_i|
        \end{equation*}
        Hence,\begin{equation}\label{Eqn11}
            \sum_{i=1}^{m-d}|\alpha_i|=\sum_{i=1}^{m-d}|\beta_i|.
        \end{equation}
        The product factor in equations (\ref{Eqn9}) and (\ref{Eqn10}) are same and its an $n(d+2)$ degree polynomial. Let its roots be $\eta_1,\eta_2,\ldots,\eta_{n(d+2)}$. From the characteristics polynomial (\ref{Eqn9}), the energy of the signed graph $\Gamma\circledast\Gamma_1$ is, \begin{equation}\label{Eqn12}
                E_{\Gamma\circledast\Gamma_1}=n\sum_{i=1}^{m-d}|\alpha_i|+\sum_{i=1}^{n(d+2)}|\eta_i|
        \end{equation}
        and from the characteristics polynomial (\ref{Eqn10}), the energy of the signed graph $\Gamma\circledast\Gamma_2$ is,
        \begin{equation}\label{Eqn13}
            E_{\Gamma\circledast\Gamma_2}=n\sum_{i=1}^{m-d}|\beta_i|+\sum_{i=1}^{n(d+2)}|\eta_i|.
        \end{equation}
        Then from the above two equations and equation~(\ref{Eqn11}), we have $E_{\Gamma\circledast\Gamma_1}=E_{\Gamma\circledast\Gamma_2}$.\\
        As $\Gamma_1$ and $\Gamma_2$ are non-co-spectral signed graphs, so $\Gamma\circledast\Gamma_1$ and $\Gamma\circledast\Gamma_2$ are non-co-spectral equienergetic.
   \end{proof}
   The above theorem can generate families of non-co-spectral equienergetic signed graphs from a given pair of non-co-spectral equienergetic signed graphs with the same coronals. The following corollary gives the existence of such non-trivial non-co-spectral equienergetic signed graphs with the same coronals.

    \begin{corollary}
        Let $\Gamma_1$ and $\Gamma_2$ be two non-co-spectral equienergetic $r$-regular signed graphs, then for any arbitrary signed graph $\Gamma$, the signed graphs $\Gamma\circledast\Gamma_1$ and $\Gamma\circledast\Gamma_2$ are also non-co-spectral equienergetic.
    \end{corollary}
    \begin{proof}
        Suppose that $\Gamma_1$ and $\Gamma_2$ are non-co-spectral equienergetic $r$-regular signed graphs. The coronal of any two $r$-regular signed graphs is known to be equal. Hence by Theorem~\ref{TH3}, the signed graphs $\Gamma\circledast\Gamma_1$ and $\Gamma\circledast\Gamma_2$ are non-co-spectral energetic signed graphs.
    \end{proof}

\section{Conclusion}
    These results establish a spectral framework for analyzing the interplay between the structural properties of the base signed graphs and their corona products.
    As an application, we showcase the practical utility of the duplication-signed corona product in generating integral and equienergetic signed graphs. These graphs are of great interest in graph theory because of their combinatorial and spectral significance. We give concrete examples to demonstrate the construction and verification of these properties.
    This work expands the study of signed graphs and their spectral properties, providing fresh opportunities to explore concepts like structural balance and graph spectra. It also shows how these ideas can be applied to create specialized graph classes, making meaningful contributions to the field of graph theory.
    
 \section*{Acknowledgment}
    We want to acknowledge the National Institute of Technology Sikkim for giving a doctoral fellowship to Bishal Sonar and Satyam Guragain.
\section*{Funding}
    The authors declare that no funds, grants, or other support were received during the preparation of this manuscript.
\section*{Data Availability}
    Data sharing does not apply to this article as no datasets were generated or analyzed during the current study.
\section*{Deceleration}
\subsection*{Conflict of interest}
    The authors have no relevant financial or non-financial interests to disclose.
\bibliographystyle{abbrv} 
\bibliography{main.bib}

\begin{thebibliography}{10}

\bibitem{adhikari2023corona}
B.~Adhikari, A.~Singh, and S.~K. Yadav.
\newblock Corona product of signed graphs and its application to modeling signed networks.
\newblock {\em Discrete Mathematics, Algorithms and Applications}, 15(01):2250062, 2023.

\bibitem{balinska2002survey}
K.~Bali{\'n}ska, D.~Cvetkovi{\'c}, Z.~Radosavljevi{\'c}, S.~Simi{\'c}, and D.~Stevanovi{\'c}.
\newblock A survey on integral graphs.
\newblock {\em Publikacije Elektrotehni{\v{c}}kog fakulteta. Serija Matematika}, pages 42--65, 2002.

\bibitem{bapat2010graphs}
R.~B. Bapat.
\newblock {\em Graphs and matrices}, volume~27.
\newblock Springer, 2010.

\bibitem{bhat2015equienergetic}
M.~A. Bhat and S.~Pirzada.
\newblock On equienergetic signed graphs.
\newblock {\em Discrete Applied Mathematics}, 189:1--7, 2015.

\bibitem{cui2012spectrum}
S.-Y. Cui and G.-X. Tian.
\newblock The spectrum and the signless laplacian spectrum of coronae.
\newblock {\em Linear algebra and its applications}, 437(7):1692--1703, 2012.

\bibitem{cvetkovic1980spectra}
D.~M. Cvetkovic, M.~Doob, and H.~Sachs.
\newblock Spectra of graphs. theory and application.
\newblock {\em Third edition}, 1980.

\bibitem{frucht1970corona}
R.~Frucht and F.~Harary.
\newblock On the corona of two graphs.
\newblock {\em Valparaiso, Chile; Ann Arbor, Mich., U.S.A.}, 1970.

\bibitem{gopalapillai2011spectrum}
I.~Gopalapillai.
\newblock The spectrum of neighborhood corona of graphs.
\newblock {\em Kragujevac journal of mathematics}, 35(37):493--500, 2011.

\bibitem{harary1953notion}
F.~Harary.
\newblock On the notion of balance of a signed graph.
\newblock {\em Michigan Mathematical Journal}, 2(2):143--146, 1953.

\bibitem{indulal2012spectrum}
G.~Indulal.
\newblock Spectrum of two new joins of graphs and infinite families of integral graphs.
\newblock {\em Kragujevac journal of mathematics}, 36(38):133--139, 2012.

\bibitem{mcleman2011spectra}
C.~McLeman and E.~McNicholas.
\newblock Spectra of coronae.
\newblock {\em Linear algebra and its applications}, 435(5):998--1007, 2011.

\bibitem{nayak2016net}
N.~G. Nayak.
\newblock {\em On net-regular signed graphs}.
\newblock Infinite Study, 2016.

\bibitem{neumaier1992horn}
A.~Neumaier.
\newblock Horn, ra; johnson, cr, topics in matrix analysis. cambridge etc., cambridge university press 1991. viii, 607 pp.,{\pounds} 45.00. isbn 0-521-30587-x.
\newblock {\em Zeitschrift Angewandte Mathematik und Mechanik}, 72(12):692--692, 1992.

\bibitem{sampathkumar1973duplicate}
E.~Sampathkumar.
\newblock On duplicate graphs.
\newblock {\em J. Indian Math. Soc}, 37:285--293, 1973.

\bibitem{article}
E.~Sampathkumar, M.~A. Sriraj, and L.~Pushpalatha.
\newblock Strong signed graph structures labeled graph structures and vertex labeled graphs.
\newblock {\em Indian J. Discrete Math}, Volume 3:15--24, 01 2017.

\bibitem{shahul2015co}
K.~Shahul~Hameed, V.~Paul, and K.~Germina.
\newblock On co-regular signed graphs.
\newblock {\em Australas J Combin}, 62(1):8--17, 2015.

\bibitem{Structuralbalanceandsignedgraphs_2025}
A.~Singh, R.~Srivastava, B.~Adhikari, and S.~K. Yadav.
\newblock Structural balance and spectral properties of generalized corona product of signed graphs.
\newblock {\em Proyecciones (Antofagasta, On line)}, 44(2):216–240, Apr. 2025.

\bibitem{sonar2023spectrum}
B.~Sonar, S.~Guragain, and R.~Srivastava.
\newblock On spectrum of neighbourhood corona product of signed graphs.
\newblock In {\em International Conference On Algebra And Its Applications}, pages 83--100. Springer, 2023.

\bibitem{varghese2017spectrum}
R.~P. Varghese and D.~Susha.
\newblock The spectrum of two new corona of graphs and its applications.
\newblock {\em International Journal of Mathematics and its Applications}, 5(4):395--406, 2017.

\end{thebibliography}
\end{document}